\def\CC         {{\mathbb C}}

\def\PP         {{\mathbb P}}

\def\dual	{{\rm \vee}}

\documentclass{amsart}
\usepackage{amssymb}
\usepackage{verbatim}
\usepackage{eucal}
\usepackage{mathrsfs}
\usepackage{graphicx}
\usepackage{psfrag}

\newtheorem{theorem}{Theorem}[section]

\newtheorem{proposition}[theorem]{Proposition}

\theoremstyle{definition}
\newtheorem{definition}[theorem]{Definition}

\theoremstyle{remark}
\newtheorem{remark}[theorem]{Remark}

\begin{document}

\title{The class of the affine line is a zero divisor in the Grothendieck ring}

\author{Lev A. Borisov}
\address{Department of Mathematics\\
Rutgers University\\
Piscataway, NJ 08854}
\email{borisov@math.rutgers.edu}

\thanks{The author was partially supported by NSF grant DMS-1201466.
}

\begin{abstract}
We show that the class  of the affine line is a zero divisor in the Grothendieck ring of algebraic varieties over complex numbers. The argument is based on the Pfaffian-Grassmannian double mirror correspondence.
\end{abstract}

\maketitle

\section{Introduction}
The Grothendieck ring $K_0(Var/\CC)$ of complex algebraic varieties is a fundamental object of algebraic geometry.
It is defined as the quotient of the group of formal integer linear combinations $\sum_{i} a_i [Z_i]$ of isomorphism classes of complex algebraic varieties modulo the relations 
$$
[Z]-[U]-[Z\backslash U]
$$
for all open subvarieties $U\subseteq Z$. The product structure is induced from the Cartesian product.

\medskip
The main result of this paper is the following.

\bigskip
{\bf Theorem \ref{main}.}
The class $L$ of the affine line is a zero divisor in the Grothendieck ring of varieties over $\CC$.

\bigskip

\medskip
The class $L=[\CC^1]$ of the affine line  plays an important role in the study of $K_0(Var/\CC)$.  For example, it has been
proved in \cite{LL} that the quotient of $K_0(Var/\CC)$ by $L$ has a natural basis indexed by 
the classes of projective algebraic varieties up to stable birational equivalence. In other instances one needs to localize $K_0(Var/\CC)$ by $L$ (see \cite{DL,Litt}) so it is important to know whether $L$ is a nonzero divisor. While it has been shown 
in \cite{Poonen} that $K_0(Var/\CC)$ is not a domain, there remained a hope that $L$ is nonetheless a non-zero-divisor in $K_0(Var/\CC)$.

\medskip
This problem was brought to our attention by an elegant recent preprint of Galkin and Shinder \cite{GSh} in which the 
authors prove that if $L$ is a nonzero divisor in $K_0(Var/\CC)$ (a weaker condition that $L^2a=0$ implies $a\in\langle L\rangle$ in fact suffices) then
a rational smooth cubic fourfold in $\PP^5$ must have its Fano variety of lines birational to a symmetric square of a K3 surface. This paper puts a dent in this approach to (ir)rationality of cubic fourfolds.

\medskip
The consequence of our construction is another important result, which was pointed to us by Evgeny Shinder. A cut-and-paste conjecture (or question) of
Larsen and Lunts
\cite[Question 1.2]{LL}  asks whether any two algebraic varieties $X$ and $Y$ with $[X]=[Y]$ in the Grothendieck ring can be cut into into disjoint unions of pairwise isomorphic locally closed subvarieties.

\bigskip
{\bf Theorem \ref{Sh}.}
The cut-and-paste conjecture of Larsen and Lunts fails.

\medskip
The negative answer to this conjecture is important in view of its potential applications to rationality of motivic zeta functions, see \cite{DL}, \cite{LL2}.

\medskip
The main idea of the proof of Theorems \ref{main} and \ref{Sh} is to compare the two sides $X_W$ and $Y_W$ of the Pfaffian-Grassmannian double mirror correspondence. These are non-birational smooth Calabi-Yau threefolds which are derived equivalent. 
There is a natural variety (a frame bundle over the Cayley hypersurface of $X_W$) whose class in the Grothendieck ring can be expressed both in terms of $[X_W]$ and in terms of $[Y_W]$. This provides a relation
$$
\Big([X_W]-[Y_W]\Big)(L^2-1)(L-1)L^7=0
$$
in the Grothendieck ring, which then implies that $L$ is a zero divisor.

\bigskip
{\bf Acknowledgements.} This paper came about as a byproduct of joint work with Anatoly Libgober on higher dimensional version of Pfaffian-Grassmannian double mirror correspondence. The author is indebted to Prof. Libgober for stimulating conversations, useful references and comments on the preliminary version of the paper. I also thank Evgeny Shinder who pointed out that the construction of the paper gives a counterexample to the cut-and-paste conjecture of Larsen and Lunts, see Theorem \ref{Sh}.

\section{The construction}
\subsection{Pfaffian and Grassmannian double mirror Calabi-Yau varieties} 
Let $V$ be a $7$-dimensional complex vector space.
Let $W\subset \Lambda^2V^\dual$ be a generic $7$-dimensional space of  skew forms on $V$. 
These data encode two smooth Calabi-Yau varieties
$X_W$ and $Y_W$ as follows.

\begin{definition}
We define $X_W$ as a subvariety of the Grassmannian $G(2,V)$ of dimension two subspaces $T_2\subset V$
which is the locus of all $T_2\in G(2,V)$ with $w\Big\vert_{T_2} = 0$ for all $w\in W$.
We define 
$Y_W$ as a subvariety of the Pfaffian variety $Pf(V)\subset \PP\Lambda^2 V$ of skew forms on $V$
whose rank is less than $6$. It is defined as the intersection of $Pf(V)$ with $\PP W\subset \PP\Lambda^2 V$.
\end{definition}

The following proposition summarizes the properties of $X_W$ and $Y_W$ that will be used later.
\begin{proposition}\label{facts}
The following statements hold for a general choice of $W$.
\begin{itemize}
\item The varieties $X_W$ and $Y_W$ are smooth Calabi-Yau threefolds.
\item The varieties $X_W$ and $Y_W$ are not isomorphic,  or even birational, to each other. 
\item All forms $\CC w\in Y_W$ have rank $4$. All forms $\CC w \in  \PP W\backslash Y_W$ have rank $6$.
\end{itemize}
\end{proposition}

\begin{proof}
Smoothness of $X_W$ and $Y_W$ has been shown by R\o dland \cite{Rodland}. 
They are not isomorphic to each other because the
ample generators $D_X$ and $D_Y$ of their respective  Picard groups have $D_X^3=42$ and $D_Y^3=14$.
The statement that $X_W$ and $Y_W$ are not birational follows 
from the fact that they are non-isomorphic Calabi-Yau threefolds with  Picard number one, see 
 \cite{BC}.

\smallskip
The statement about the rank of the forms follows from the fact that $W$ is generic, since the locus of rank $2$ forms in
$\PP\Lambda^2V^\dual$ is of codimension $10$. Alternatively, if $\CC w\in Y_W$ has rank $2$, then $Y_W$ 
is automatically singular at  $\CC w$.
\end{proof}

\begin{remark}
The varieties $X_W$ and $Y_W$ are double-mirror to each other, in the sense that they have the same mirror family.
This is just a heuristic statement, but it does indicate that geometry of $X_W$ is intimately connected to that of $Y_W$.
For example, it was shown independently in  \cite{BC}  and  \cite{Kuznetsov} that $X_W$ and $Y_W$ 
have equivalent derived categories.
\end{remark}

\subsection{Cayley hypersurface and its frame bundle.}
The main technical tool of this paper is the so-called Cayley hypersurface of $X_W$. It is the hypersurface in
$G(2,V)\times \PP W$ which consists of pairs $(T_2, \CC w)$ with the property $w\Big\vert_{T_2}=0$.
The class of $X_W$ in the Grothendieck ring of varieties over $\CC$ is related to that of $H$ as follows.
\begin{proposition}\label{HX}
The following equality holds in the Grothendieck ring.
$$
[H] = [G(2,7)] [\PP^5] +[X_W] L^6
$$
\end{proposition}

\begin{proof}
Consider the projection of $H$ onto $G(2,V)$. The restriction of this map to the preimage of $X_W$ is a trivial
fibration with fiber $\PP W=\PP^6$. The restriction of it to the complement  of $X_W$ is a Zariski locally trivial
fibration with fiber $\PP^5$. Indeed, the hyperplanes of $w$ that vanish on a given $T_2$ can be 
Zariski locally identified with a fixed $\PP^5$ by projecting from a fixed point in $\PP W$. 
This gives 
$$
[H] =[X_W] [\PP^6] + ([G(2,7)]-[X_W])[\PP^5] =  [G(2,7)] [ P^5] +[X_W] ([\PP^6]-[\PP^5])
$$
which proves the claim.
\end{proof}

\begin{remark}
In the proof of Proposition \ref{HX} we used the statement that for  a Zariski locally trivial fibration $Z\to B$ with fiber $F$
there holds $[Z]=[B][F]$ in $K_0(Var/\CC)$. We will use this statement repeatedly in the subsequent arguments.
\end{remark}

We can project the Cayley hypersurface $H$ onto the second factor $\pi:H\to \PP W$. We will have different fibers depending on whether the image lies in $Y_W$ or not. While we would like to say that the restriction of $\pi$ to
the preimages of $Y_W$ and its complement are Zariski locally trivial, we do not know if this is true or not.
So instead of using $H$ itself we will pass to the frame bundle $\tilde H$ over $H$.

\begin{definition}
We denote by $\tilde H$ the frame bundle of $H$, i.e. the space of triples $(v_1,v_2,w)$ where $v_1$ and $v_2$
are linearly independent vectors in $V$ and $w$ is an element of $\PP W$ such that $w(v_1,v_2)=0$.
\end{definition}

\begin{remark}
Since $\tilde H$ is the frame bundle of the Zariski locally trivial vector bundle (pullback of the tautological subbundle 
on $G(2,V)$) on $H$, the fibration $\tilde H\to H$ is Zariski locally trivial. An easy calculation shows that 
\begin{equation}\label{tilde}
[\tilde H] = [H] (L^2-1)(L^2-L)
\end{equation}
in the Grothendieck ring.
\end{remark}

We now consider the projection $\tilde H\to \PP W$. Notice that we have 
\begin{equation}\label{disk}
\tilde H = \tilde H_{1} \sqcup \tilde H_2
\end{equation}
where $\tilde H_{1}$ is the preimage of $Y_W$ and $\tilde H_2$ is the preimage of its complement in $\PP W$.

\begin{proposition}\label{H1}
The following equality holds in the Grothendieck ring.
$$
[\tilde H_1] = [Y_W]\Big((L^3-1)(L^7-L) + (L^7-L^3)(L^6-L)\Big)
$$
\end{proposition}

\begin{proof}
There is a subvariety $\tilde H_{1,1}$ in $\tilde H_1$ given by the condition $v_1\in Ker(w)$.
Forgetting $v_2$ realizes $\tilde H_{1,1}$  as a Zariski locally trivial fibration with fiber $\CC^7-\CC$ over the space of pairs  $(v_1,w)$ with $v_1\in Ker(w),~v_1\neq 0$. This in turn is a Zariski locally trivial fibration over
$Y_W$ with fiber $(\CC^3-{\rm pt})$, since all $\CC w\in Y_W$ have rank $4$. Putting all this together, we have 
$$
[\tilde H_{1,1}] = [Y_W](L^3-1)(L^7-L) 
$$
in the Grothendieck ring. Similarly, the complement $\tilde H_{1,2}$ of $\tilde H_{1,1}$ in $\tilde H_1$ 
satisfies
$$
[\tilde H_{1,2}] = [Y_W](L^7-L^3)(L^6-L).
$$
Indeed, $\tilde H_{1,2}$  forms a vector bundle of rank $6$ over the space of pairs 
$(v_1,w)$, since the condition $w(v_1,v_2)=0$ is now nontrivial.
The result of the proposition now follows from $[\tilde H_1]=[\tilde H_{1,1}]+[\tilde H_{1,2}]$.
\smallskip
\end{proof}

\begin{proposition}\label{H2}
The following equality holds in the Grothendieck ring.
$$
[\tilde H_2] =\Big([\PP^6]- [Y_W]\Big)\Big((L-1)(L^7-L) + (L^7-L)(L^6-L)\Big)
$$
\end{proposition}

\begin{proof}
The argument is completely analogous to that  of Proposition \ref{H1}. The only difference is 
that a form $\CC w\not\in Y_W$ has rank $6$ and thus a one-dimensional kernel.
\end{proof}

As  a corollary of Propositions \ref{H1} and \ref{H2} we get the formula for $[\tilde H]$.
\begin{proposition}\label{H}
The following equality holds in the Grothendieck ring.
$$
[\tilde H] =[\PP^6](L^7-L)(L^6-1)+ [Y_W](L^2-1)(L-1)L^7
$$
\end{proposition}

\begin{proof}
This follows immediately from \eqref{disk} and Propositions \ref{H1} and \ref{H2}.
\end{proof}

\subsection{Main theorem.}
We are now ready to prove our main result. We start with the following formula
derived from the calculations of the previous subsection.
\begin{proposition}\label{zero}
The following equality  holds in the Grothendieck ring.
$$
\Big([X_W]-[Y_W]\Big)(L^2-1)(L-1)L^7=0
$$
\end{proposition}

\begin{proof}
We use Proposition \ref{H}  and Proposition  \ref{HX} with equation \eqref{tilde} to get
expressions for $[\tilde H]$, in terms of $[Y_W]$ and $[X_W]$ respectively.
By subtracting one from the other we get
$$
\Big([X_W]-[Y_W]\Big)(L^2-1)(L-1)L^7= [\PP^6](L^7-L)(L^6-1)-[G(2,7)][\PP^5](L^2-1)(L^2-L)
$$
which then equals zero in view of $[G(2,7)](L^2-1)(L^2-L) = (L^7-1)(L^7-L)$ and 
$[\PP^6](L^6-1)=[\PP^5](L^7-1)$.
\end{proof}

\begin{theorem}\label{main}
The class $L$ of the affine line is a zero divisor in the Grothendieck ring of varieties over $\CC$.
\end{theorem}

\begin{proof}
In view of Proposition \ref{zero}, it suffices to show that 
$$
\Big([X_W]-[Y_W]\Big)(L^2-1)(L-1)
$$
is a nonzero element of the Grothendieck ring. In fact, we can argue that it is a nonzero element modulo $L$.
Indeed, if it were zero modulo $L$, this would mean that $[X_W]=[Y_W]\mod L$. This implies that $X_W$
is stably birational to $Y_W$, by \cite{LL}. This means that for some $k\geq 0$ the varieties $X_W\times \PP^k$
and $Y_W\times \PP^k$ are birational to each other. We now consider the MRC fibration \cite{KMM}, 
which is a birational invariant of an algebraic variety. Importantly, if $X$ is not uniruled (for example a Calabi-Yau variety)
then the base of the MRC fibration 
of $X\times \PP^k$ is $X$. Thus, birationality of  $X_W\times \PP^k$ and  $Y_W\times \PP^k$ implies
birationality of $X_W$ and $Y_W$, which is known to be false, see Proposition \ref{facts}.
\end{proof}

It was observed by Evgeny Shinder that the construction of this paper provides a negative answer to the cut-and-paste question 
of Larsen and Lunts \cite[Question 1.2]{LL} which asks whether any two varieties with equal classes in the Grothendieck ring 
can be cut up into isomorphic pieces.
\footnote{Another counterexample to the question was recently announced by Ilya Karzhemanov in \cite{Karzhemanov}.}
\begin{theorem}\label{Sh}
The cut-and-paste conjecture of Larsen and Lunts fails.
\end{theorem}

\begin{proof}
The equality 
$$[X_W](L^2-1)(L-1)L^7= [Y_W](L^2-1)(L-1)L^7$$
implies that trivial $GL(2,\CC)\times \CC^6$ bundles over $X_W$ and $Y_W$ have the same class in the Grothendieck ring. If it were
possible to cut them into unions of isomorphic varieties, then $X_W\times GL(2,\CC)\times \CC^6$ would be birational
to $Y_W\times GL(2,\CC)\times \CC^6$. This implies that $X_W$ and $Y_W$ are stably birational, and thus birational, in contradiction 
with Proposition \ref{facts}.
\end{proof}

\begin{remark}
Our method works over any field of characteristic zero. It does not appear to work in positive characteristics, since results of
\cite{LL} are based on \cite{AKMW} which in turn relies on the resolution of singularities.
\end{remark}

\end{document}